\documentclass[10pt,a4paper]{article}

\usepackage{amsmath,amsthm}
\usepackage[numbers]{natbib}
\usepackage{microtype}
\usepackage[bitstream-charter]{mathdesign}

\usepackage{oldgerm}
\DeclareFontFamily{U}{yswab}{}	
\DeclareFontShape{U}{yswab}{m}{n}{<-> yswab}{}

\usepackage{bm}
\usepackage[final]{showkeys}

\numberwithin{equation}{section}
\usepackage[hmargin=2.5cm,vmargin={2cm,3.5cm}]{geometry}
\allowdisplaybreaks

\newtheorem{theorem}{Theorem}[section]
\newtheorem{remark}{Remark}[section]
\newtheorem{definition}{Definition}[section]
\newtheorem{proposition}{Proposition}[section]

\title{A Generalization of the Space-Fractional Poisson Process and its Connection
	to some L\'evy Processes}
	
\date{}

\author{$\text{Federico Polito}^1$ \& $\text{Enrico Scalas}^2$\\
	\footnotesize{${}^1$Dipartimento di Matematica \emph{G.~Peano}, Universit\`a degli Studi di Torino, Italy}\\
	\footnotesize{${}^2$Department of Mathematics, University of Sussex, Brighton, UK}}

\begin{document}

	\maketitle

	\begin{abstract}
			\noindent This paper introduces a generalization of the so-called space-fractional Poisson process by extending the difference operator
			acting on state space present in the associated difference-differential equations to a much more general form.
			It turns out that this generalization can be put in relation to a specific subordination of a homogeneous Poisson process
			by means of a subordinator for which it is possible to express the characterizing L\'evy measure explicitly.
			Moreover, the law of this subordinator solves a one-sided first-order differential equation in which a
			particular convolution-type integral operator appears, called Prabhakar derivative.
			In the last section of the paper, a similar model is introduced in which the Prabhakar derivative also acts in time.
			In this case, too, the probability generating function of the corresponding process and the probability
			distribution are determined.

			\medskip

			
	\end{abstract}

	\section{Introduction and background}

		In the last decade, it became apparent that several phenomena that can be modeled in terms of point processes are non Poissonian in nature 	
		(see \citep{barabasi,jiang} as examples). As a consequence, there has been an increased interest in generalizing the Poisson process $N(t)$.
		The Poisson process is a counting process with many nice properties. It is a L\'evy process and, therefore, its increments are
		time-homogeneous and independent. It is a renewal process, meaning that the sojourn times between points (or events) are
		independent and identically distributed following the exponential distribution. It is a birth-death Markov process and its
		counting probability obeys the forward Kolmogorov equation
		\begin{equation}
			\label{poisson}
 			\frac{\textup d}{\textup dt} p_k (t) = -\lambda p_k (t) + \lambda p_{k-1} (t),
		\end{equation}
		where $p_k (t) = \mathbb{P}\{N(t) = k\}$, $k \geq 0$, $t \geq 0$ are the state probabilities of the Poisson process
		and $\lambda$ is the rate of the Poisson process. There are many possible ways to generalize this process. We are interested
		in the so-called fractional generalizations of the Poisson process where either the derivative
		on the left-hand side or the difference equation on the right-hand side of \eqref{poisson} are replaced by suitable fractional operators.

		For instance, if one keeps the renewal property and considers sojourn times such that
		$\mathbb{P}\{N_\beta (t) =0\} = E_\beta ( -t^\beta)$ for $0 < \beta < 1$, where
		$E_\beta (z) = \sum_{k=0}^\infty z^k/\Gamma(k \beta +1)$
		is the one-parameter Mittag-Leffler function, one gets the renewal fractional Poisson process $N_\beta(t)$
		discussed in \citep{mainardivietnam}. This leads 
		to the equation for $p_k (t) = \mathbb{P}\{N_\beta (t) = k\}$
		\citep{laskin},
		\begin{equation}
			\frac{\textup d^\beta}{\textup dt^\beta} p_k (t) = -\lambda p_k (t) + \lambda p_{k-1} (t),
			\qquad  k \geq 0, \: t \geq 0,
		\end{equation}
		(with $\lambda=1$)
		where $\textup d^\beta/ \textup dt^\beta$ is the so-called Caputo derivative, a pseudo-differential operator defined as 
		\begin{equation}
			\label{caputo}
			\frac{\textup d^\beta}{\textup d t^\beta} f(t) = \frac{1}{\Gamma (1 - \beta)} \int_0^t (t - u)^{-\beta} 
			\frac{\textup d f(u)}{\textup d u} \, \textup du.  
		\end{equation}
		The renewal fractional Poisson process is not a L\'evy process \citep{kaizoji}.

		Another possibility is generalizing the Poisson process via the so-called space-fractional Poisson process studied
		in \citep{orsingher}. The space-fractional Poisson process is in practice a homogeneous Poisson process
		subordinated to an independent stable subordinator. We describe this model in details in Section \ref{spafra}.
		Here we aim at a further generalization of this process preserving the L\'evy property.
		This generalized process is constructed by means of a superposition
		of suitably weighted independent space-fractional Poisson processes. The resulting process is then subordinated
		by means of a random time process in order to account for the modelization of a possible irregular flow of time.
		The complete construction of the model is done in Section \ref{genspafra}.

		The motivation at the basis of this study lies in the importance that superposition of, possibly dependent, point processes plays
		in applied sciences. An important example concerns the modelling of neurons' incoming signals. It is commonly accepted that
		each neuron obtains information from the neighbouring
		neurons in the form of \emph{spike trains} (i.e.\ signals with powerful bursts), closely resembling
		realizations of stochastic point processes.
		Therefore, each neuron, receives a superposition of, possibly rescaled, point processes and its subsequent behaviour depends
		on the characteristics of this afferent combined input signal.
		Based on classical result contained in \citep{cox1954superposition,grige, franken, cinlar}
		(roughly saying that a superposition of sufficiently sparse independent point processes converges to a Poisson process),
		the input signal was considered in this applied literature to be well approximated by a Poisson process \citep[see e.g.][]{hohn, shimokawa}.
		However, it was later shown that the above result does not always apply to superposition of signals from neurons' activity
		and that experimental evidence deviates from a Poissonian structure \citep{lindner, cat, deger}.
		Recently, the study of weak convergence of superposition of point processes has regained interest
		\citep[see e.g.][and the references therein]{chen} showing that different behaviours are possible.
		Within this framework, we can consider the model we are going to describe as a weighted finite superposition of independent space-fractional
		Poisson processes (each of them generalizing the homogeneous Poisson process but also admitting the possibility
		of jumps of any integer order). Each of these space-fractional Poisson processes can be thought to model different groups
		of neurons (different areas of the brain) acting together with simultaneous spikes giving rise to the non-unitary jumps.
		Notice, finally, that the space-fractional Poisson process is a non-renewal process
		and so is the generalized space-fractional Poisson process. This is a key feature for the combined neurons' input signal as
		it is remarked in \citep{lindner}.
		
		For the sake of clarity and simplicity, we first recall some selected
		basic mathematical results regarding subordinators
		which will be useful in the following. Section \ref{lenovo} presents
		the construction of a random time-change by means of independent subordinators
		and tempered subordinators. This is of fundamental importance for the definition
		of the generalized space-fractional Poisson process which will be carried out in Section \ref{app}.
		 
		We start thus by recalling some basic facts on subordinators and tempered stable subordinators.
		The reader can refer to \citep{bertoin} or \citep{kyprianou} for a more in-depth explanation.
		We recall that a subordinator is an increasing L\'evy process
		defined as follows.
		\begin{definition}
			Given a filtered probability space $(\Omega, \mathcal{F}, \mathfrak{F}, \mathbb{P})$,
			where $\mathfrak{F} = (\mathcal{F}_t)_{t\ge 0}$ is the associated right-continuous filtration,
			the process $S_t$, $t \ge 0$, adapted to $\mathfrak{F}$, starting from zero, and with increasing paths, is called a
			subordinator if it has independent and stationary increments or, equivalently, $\forall s > t \ge 0$,
			$S_t-S_s$ is independent of $\mathcal{F}_s$ and $S_t-S_s = S_{t-s}$ in distribution.
		\end{definition}
		In this paper we consider only strict subordinators that is those with infinite lifetime
		and with the only cemetery state ($\infty$) located at $t=\infty$.
		For this special class of subordinators a simplified version of the well-known
		L\'evy--Khintchine formula for general L\'evy processes holds. In particular, the following theorem
		gives a characterization of subordinators in terms of Laplace transforms of their
		one dimensional law. Let us first recall that, thanks to stationarity and independence of
		the increments, we have that
		\begin{align}
			\mathbb{E} \exp ( -\mu S_t ) = \exp \left[ - t \Phi (\mu) \right], \qquad t \ge 0, \: \mu \ge 0,
		\end{align}
		where $\Phi(\mu)$ is called the Laplace exponent of the process $S_t$, $t \ge 0$.
		Different expressions for $\Phi(\mu)$ give rise to different subordinators, but not
		any functional form is allowed. This is exactly what the L\'evy--Khintchine formula
		for subordinators tells us.

		\begin{theorem}
			\label{lkt}
			Any function $\Phi(\mu)$, $\mu \ge 0$, that can be put in the unique form
			\begin{align}
				\label{lk}
				\Phi(\mu) = b \mu + \int_0^\infty \left(1-e^{-\mu x}\right) m(\mathrm dx), \qquad b \ge 0,
			\end{align}
			where $m(\mathrm dx)$ is a measure concentrated on $(0,\infty)$ such that
			$\int_0^\infty (1 \wedge x) m(\mathrm dx) < \infty$, is the Laplace exponent of a strict
			subordinator.
			Conversely, if $\Phi(\mu)$ is the Laplace exponent of a strict subordinator,
			there exist a non negative number $b$ and a unique measure $m$ with
			$\int_0^\infty (1 \wedge x) m(\mathrm dx) < \infty$ such that \eqref{lk} holds true.
		\end{theorem}

		For a detailed proof see \citep{bertoin}. The measure $m$ is called the L\'evy measure
		and $b$ is the drift associated to the strict subordinator $S_t$, $t \ge 0$.
		For the sake of simplicity, in the following we shall use the term subordinator
		meaning in fact a strict subordinator.
		
		A particularly simple example of a subordinator is the so-called stable subordinator.
		This is characterized by the Laplace exponent $\Phi(\mu) = \mu^\alpha$, $\alpha \in (0,1)$,
		corresponding to a L\'evy measure $m (\mathrm dx) = \left[ \alpha/\Gamma(1-\alpha) \right]
		x^{-1-\alpha} \mathrm dx$.
		Note that the case $\alpha=1$ is omitted as it is trivial.
		The one-parameter $\alpha$-stable subordinator is at the basis of the probabilistic
		theory of anomalous diffusion based on subordination. In fact, it is
		a building block for the numerous processes connected to fractional
		evolution equations and their generalizations \citep{meerschaert, dovidio}
		and also for other processes connected for example to point processes \citep{orsingher,beghin}.
		Moreover, the importance of stable subordinators stems also from the fact that
		they are scaling limits of some totally skewed generalized random walks \citep{mee}.
		Given a subordinator $S_t$, $t \ge 0$, it is possible to define its right-inverse
		process \citep{bingham} as
		\begin{align}
			E_t = \inf \{ w > 0 \colon S_w > t \}, \qquad t \ge 0.
		\end{align}
		The process $E_t$ is non-Markovian with non-stationary and
		dependent increments \citep{veillette, lageras}.
		
		Let us now introduce some basic facts on tempered subordinators, in particular
		for tempered stable-subordinators. First of all, let us refer the reader to the paper
		\citep{rosinski} for a complete and detailed account on the general theory
		of tempered stable processes. However, it is interesting to note that tempered models
		already appeared in the literature (amongst others, for example, the KoBoL model \citep{kop, bl}).
		The class of tempered stable subordinators has been introduced in order
		to have processes possessing nicer properties than those of standard stable subordinators.
		In practice the L\'evy measure of a stable subordinator is exponentially tempered, thus
		obtaining
		\begin{align}
			m(\mathrm dx) = e^{-\xi x} \frac{\alpha}{\Gamma(1-\alpha)} x^{-1-\alpha} \mathrm dx, \qquad
			\alpha \in (0,1), \: \xi > 0.
		\end{align}
		This simple operation produces desirable results as seen immediately
		by realizing that the Laplace exponent in this case can be written as
		$\Phi (\mu) = (\xi+\mu)^\alpha - \xi^\alpha$.
		For more information regarding tempered stable subordinators and associated differential
		equations, see \citep{baeumer}.
		In Section \ref{lenovo}, we introduce a process which has both the standard
		stable subordinator and the tempered stable subordinator as special cases and we study
		its properties. In Section \ref{app}, instead, we describe connections of the introduced
		process with some difference-differential equations involving a generalized
		fractional difference operator acting in space. An interesting special
		case related to these equations is that regarding the so-called
		space-fractional Poisson process \citep{orsingher}.
		In the last Section \ref{verylast}, we present a further
		generalization leading to a process which is no longer a L\'evy process. This
		generalization is based on the study of similar difference-differential
		equations but involving the so-called regularized Prabhakar derivative
		in time that generalizes the Caputo derivative in time.

	\section{A subordinated sum of independent stable subordinators}
	
		\label{lenovo}
		Let us consider the filtered probability space $(\Omega, \mathcal{F}, \mathfrak{F}, \mathbb{P})$
		and the process
		\begin{align}
			{}_\eta \mathcal{V}_t^{\nu,n} = \sum_{r=1}^n \binom{n}{r}^{\frac{1}{\nu r}} \! \!
			\eta^{\frac{n-r}{\nu r}} V_t^{\nu r}
			= \sum_{r=1}^n V_{\binom{n}{r}\eta^{n-r} t}^{\nu r}, \quad t \ge 0,  \, n \in \mathbb{N}, \, \eta>0, \,
			\nu r \in (0,1), \, r=1, \dots, n,
		\end{align}
		where $V_t^{\nu r}$, $t \ge 0$, $r=1, \dots, n$, are $n$ $\mathfrak{F}$-adapted independent $\nu r$-stable subordinators.
		Let us further consider a positive real parameter $\delta$ such that $\lceil \delta \rceil = n$
		and the $\mathfrak{F}$-adapted tempered $\delta/n$-stable subordinator $\mathscr{V}_t^{\delta/n}$,
		$t \ge 0$,
		evaluated at $\eta^n$, i.e.\ with Laplace exponent
		\begin{align}
			\label{garlach}
			\Phi(\mu) = \left( \eta^n + \mu \right)^{\delta/n} - \eta^\delta, \qquad \delta > 0, \: \eta > 0.
		\end{align}

		\begin{proposition}
			The Laplace transform of ${}_\eta \mathcal{V}_t^{\nu,n}$, $t \ge 0$, reads
			\begin{align}
				\label{rossi}
				\mathbb{E} \exp \left( -\mu \, {}_\eta \mathcal{V}_t^{\nu,n} \right)
				= \exp \left( -t \left[ \left( \eta+\mu^\nu \right)^n -\eta^n \right] \right),
				\qquad t \ge 0, \: \mu > 0.
			\end{align}
			\begin{proof}
				Formula \eqref{rossi} can be proven by noticing that
				\begin{align}
					\mathbb{E} \exp \left( -\mu \: {}_\eta \mathcal{V}_t^{\nu,n} \right)
					= \prod_{r=1}^n \mathbb{E} \exp \left( - \mu \binom{n}{r}^{\frac{1}{\nu r}}
					\eta^{\frac{n-r}{\nu r}} V_t^{\nu r} \right)
				\end{align}
				and by recalling that $\mathbb{E} \exp \left( -\mu V_t^\beta \right) = \exp \left( -t \mu^\beta \right)$.
				We thus obtain
				\begin{align}
					\mathbb{E} \exp \left( -\mu \: {}_\eta \mathcal{V}_t^{\nu,n} \right)
					= \prod_{r=1}^n \exp \left( -t \binom{n}{r} \eta^{n-r} \mu^{\nu r} \right)
				\end{align}
				and hence the claimed result.
			\end{proof}
		\end{proposition}

		Let us now consider the probability space $(\Omega, \mathcal{F}, \mathbb{P})$ and
		the subordinated filtration
		$\mathfrak{G} = \left( \mathcal{G}_t \right)_{t \ge 0}
			= \bigl( \mathcal{F}_{\mathscr{V}_t^{\delta/n}} \bigr)_{t \ge 0}$.
		In the following, our interest will focus on the subordinated and $\mathfrak{G}$-adapted
		process
		\begin{align}
			\label{proc}
			{}_\eta \mathfrak{V}_t^{\nu,\delta} = {}_\eta \mathcal{V}_{\mathscr{V}_t^{\delta/n}}^{\nu,n},
			\qquad t \ge 0, \: \delta > 0, \: \nu \in (0,1), \: \eta >0.
		\end{align}

		\begin{proposition}
			\label{mamady}
			The Laplace transform of the process ${}_\eta \mathfrak{V}_t^{\nu,\delta}$, $t \ge 0$, can be
			written as
			\begin{align}
				\label{proc2}
				\mathbb{E} \exp \left( -\mu \: {}_\eta \mathfrak{V}_t^{\nu,\delta} \right)
				= \exp \left( -t \left[ \left( \eta+\mu^\nu \right)^\delta -\eta^\delta \right] \right),
				\qquad t \ge 0, \: \mu > 0.
			\end{align}
			\begin{proof}
				By definition \eqref{proc} we can write that
				\begin{align}
					\mathbb{E} \exp \left( -\mu \: {}_\eta \mathfrak{V}_t^{\nu,\delta} \right)
					& = \int_0^\infty \mathbb{E} \exp \left( -\mu \: {}_\eta \mathcal{V}_s^{\nu,n} \right)
					\mathbb{P} \left( \mathscr{V}_t^{\delta/n} \in \mathrm ds \right) \\
					& = \int_0^\infty \exp \left( -s \left[ \left( \eta+\mu^\nu \right)^n -\eta^n \right] \right)
					\mathbb{P} \left( \mathscr{V}_t^{\delta/n} \in \mathrm ds \right). \notag
				\end{align}
				Then by using \eqref{garlach} we directly arrive at \eqref{proc2}.
			\end{proof}
		\end{proposition}

		\begin{remark}
			Plainly, when $\delta=n=1$, the process \eqref{proc} coincides with a standard $\nu$-stable
			subordinator, while for $\nu \to 1$, $\delta \in (0,1)$, it is a tempered $\delta$-stable
			subordinator. As an example, we also note that for $\delta=2$, $\nu \in (0,1/2)$,
			${}_\eta \mathfrak{V}_t^{\nu,\delta} = V_t^{2\nu} + V_{2 \eta t}^\nu$ is the rescaled
			sum of two independent stable subordinators.
		\end{remark}
		
		In the following theorem we prove that the $\mathfrak{G}$-adapted process ${}_\eta \mathfrak{V}_t^{\nu,\delta}$,
		$t \ge 0$, is in fact a subordinator and derive its associated L\'evy measure.
		
		\begin{theorem}
			\label{measura}
			The $\mathfrak{G}$-adapted process ${}_\eta \mathfrak{V}_t^{\nu,\delta}$, $t \ge 0$, is a
			subordinator. Furthermore, its associated L\'evy measure is
			\begin{align}
				\label{lamisura}
				m(\mathrm dx) = \int_0^\infty e^{-\eta^n t} \mathbb{P} \left( {}_\eta \mathcal{V}_t^{\nu,n}
				\in \mathrm dx \right) \frac{\delta}{n} \frac{t^{-\left( 1+\frac{\delta}{n} \right)}}{
				\Gamma \left( 1-\frac{\delta}{n} \right)} \mathrm dt.
			\end{align}
			\begin{proof}
				In order to prove the statement of the theorem we make use of the L\'evy--Khintchine
				formula for subordinators. The Laplace exponent associated to 
				${}_\eta \mathfrak{V}_t^{\nu,\delta}$, $t \ge 0$, that is
				$\Phi(\mu) = \left( \eta+\mu^\nu \right)^\delta -\eta^\delta$,
				can be retrieved with the following steps.
				\begin{align}
					& \int_0^\infty \left( 1-e^{-\mu x} \right)
					\int_0^\infty e^{-\eta^n t} \mathbb{P} \left( {}_\eta \mathcal{V}_t^{\nu,n}
					\in \mathrm dx \right) \frac{\delta}{n} \frac{t^{-\left( 1+\frac{\delta}{n} \right)}}{
					\Gamma \left( 1-\frac{\delta}{n} \right)} \mathrm dt \\
					& = \int_0^\infty \mathrm dt \, e^{-\eta^n t} \frac{\delta}{n}
					\frac{t^{-\left( 1+\frac{\delta}{n} \right)}}{
					\Gamma \left( 1-\frac{\delta}{n} \right)} \int_0^\infty \left( 1-e^{-\mu x} \right)
					\mathbb{P} \left( {}_\eta \mathcal{V}_t^{\nu,n} \in \mathrm dx \right) \notag \\
					& = \int_0^\infty \mathrm dt \, e^{-\eta^n t} \frac{\delta}{n}
					\frac{t^{-\left( 1+\frac{\delta}{n} \right)}}{
					\Gamma \left( 1-\frac{\delta}{n} \right)}
					\left[ 1-\int_0^\infty e^{-\mu x} \mathbb{P}
					\left( {}_\eta \mathcal{V}_t^{\nu,n} \in \mathrm dx \right) \right] \notag \\
					& = \int_0^\infty \mathrm dt \, e^{-\eta^n t} \frac{\delta}{n}
					\frac{t^{-\left( 1+\frac{\delta}{n} \right)}}{
					\Gamma \left( 1-\frac{\delta}{n} \right)}
					\left( 1-e^{-t \left[ \left( \eta+\mu^\nu \right)^n -\eta^n \right]} \right) \notag \\
					& = \int_0^\infty \mathrm dt \left( e^{-\eta^n t} - e^{-t\left( \eta+\mu^\nu \right)} \right)
					\frac{\delta}{n}\frac{t^{-\left( 1+\frac{\delta}{n} \right)}}{
					\Gamma \left( 1-\frac{\delta}{n} \right)} \notag \\
					& = \int_0^\infty \mathrm dt \left( 1 - e^{-t\left( \eta+\mu^\nu \right)} \right)
					\frac{\delta}{n}\frac{t^{-\left( 1+\frac{\delta}{n} \right)}}{
					\Gamma \left( 1-\frac{\delta}{n} \right)}
					- \int_0^\infty \mathrm dt \left( 1 - e^{-\eta^n t} \right)
					\frac{\delta}{n}\frac{t^{-\left( 1+\frac{\delta}{n} \right)}}{
					\Gamma \left( 1-\frac{\delta}{n} \right)} \notag \\
					& = \frac{\delta}{n} \frac{1}{\Gamma \left( 1-\frac{\delta}{n} \right)}
					\left[ \int_0^\infty \int_0^t \left( \eta+\mu^\nu \right)^n
					e^{-y \left( \eta+\mu^\nu \right)^n} \mathrm dy \, t^{-1-\frac{\delta}{n}} \mathrm dt
					- \int_0^\infty \int_0^t \eta^n e^{-y \eta^n} \mathrm dy \, t^{-1-\frac{\delta}{n}}
					\mathrm dt  \right] \notag \\
					& = \frac{\delta}{n} \frac{1}{\Gamma \left( 1-\frac{\delta}{n} \right)}
					\left[ \int_0^\infty \int_y^\infty \left( \eta+\mu^\nu \right)^n
					e^{-y \left( \eta+\mu^\nu \right)^n} \mathrm dy \, t^{-1-\frac{\delta}{n}} \mathrm dt
					- \int_0^\infty \int_y^\infty \eta^n e^{-y \eta^n} \mathrm dy \, t^{-1-\frac{\delta}{n}}
					\mathrm dt  \right] \notag \\
					& = \frac{\delta}{n} \frac{1}{\Gamma \left( 1-\frac{\delta}{n} \right)}
					\int_0^\infty \mathrm dy \left( \eta+\mu^\nu \right)^n e^{-y \left( \eta+\mu^\nu \right)^n}
					y^{-\frac{\delta}{n}} \frac{n}{\delta}
					-\frac{\delta}{n} \frac{1}{\Gamma \left( 1-\frac{\delta}{n} \right)}
					\int_0^\infty \mathrm dy \, \eta^n e^{-y \eta^n}
					y^{-\frac{\delta}{n}} \frac{n}{\delta} \notag \\
					& = \frac{1}{\Gamma \left( 1-\frac{\delta}{n} \right)} \int_0^\infty
					\mathrm dz \, e^{-z} \left( \frac{z}{\left( \eta+\mu^\nu \right)^n} \right)^{-\frac{\delta}{n}}
					- \frac{1}{\Gamma \left( 1-\frac{\delta}{n} \right)} \int_0^\infty \mathrm dz \,
					e^{-z} \left( \frac{z}{\eta^n} \right)^{-\frac{\delta}{n}} \notag \\
					& = \left( \eta+\mu^\nu \right)^\delta -\eta^\delta. \notag
				\end{align}
				From the above computation and Theorem \ref{lkt}, the claimed form for the L\'evy measure easily follows.
			\end{proof}
		\end{theorem}		
		
		\begin{remark}
			Note that the determined measure can also be expressed by means of Riemann--Liouville
			fractional derivatives. Indeed, since \citep[formula (2.117)]{podlubny}
			\begin{align}
				- \frac{\delta}{n} \frac{t^{-\left( 1+\frac{\delta}{n} \right)}}{\Gamma
				\left( 1-\frac{\delta}{n} \right)} = \frac{t^{-\left( 1+\frac{\delta}{n} \right)}}{\Gamma
				\left( 1-\left( 1+\frac{\delta}{n} \right) \right)} =
				\left( D_{0+}^{1+\frac{\delta}{n}} 1 \right) (t),
			\end{align}
			where
			\begin{align}
				\left( D_{a+}^\beta f \right) (t) = \frac{1}{\Gamma(m-\beta)}
				\frac{\mathrm d^m}{\mathrm d t^m} \int_a^\infty \frac{f(y)}{(t-y)^{\beta-m+1}} \mathrm dy,
				\qquad \beta > 0, \: m = [\beta]+1,
			\end{align}
			is the Riemann--Liouville fractional derivative \citep{podlubny}, we obtain readily that
			\begin{align}
				m (\mathrm dx) = - \int_0^\infty e^{-\eta^n t} \mathbb{P} \left( {}_\eta \mathcal{V}_t^{\nu,n}
				\in \mathrm dx \right) \left( D_{0+}^{1+\frac{\delta}{n}} 1 \right) (t) \, \mathrm dt.
			\end{align}
		\end{remark}
		
		In order to study the operator associated to the subordinator ${}_\eta \mathfrak{V}_t^{\nu,\delta}$ we start from the
		Laplace transform
		\begin{align}
			\mathbb{E} e^{-\mu \, {}_\eta \mathfrak{V}_t^{\nu,\delta}} = e^{-t \left[ (\eta+\mu^\nu)^\delta -\eta^\delta \right]},
			\qquad t \ge 0, \: \mu > 0.
		\end{align}
		By taking the derivative with respect to time we obtain
		\begin{align}
			\label{bilbao}
			\frac{\mathrm d}{\mathrm dt} \mathbb{E} e^{-\mu\, {}_\eta \mathfrak{V}_t^{\nu,\delta}}
			= - \left[ (\eta+\mu^\nu)^\delta -\eta^\delta \right] \mathbb{E} e^{-\mu \, {}_\eta \mathfrak{V}_t^{\nu,\delta}}.
		\end{align}
		From Theorem \ref{measura} we know that
		\begin{align}
			\left[ (\eta+\mu^\nu)^\delta -\eta^\delta \right] = \int_0^{\infty} (1-e^{-\mu x})m(\mathrm dx),
		\end{align}
		where $m(\mathrm dx)$ is given by \eqref{lamisura}. Hence,
		\begin{align}
			\left[ (\eta+\mu^\nu)^\delta -\eta^\delta \right] \mathbb{E} e^{-\mu \, {}_\eta \mathfrak{V}_t^{\nu,\delta}}
			= \int_0^\infty \left( \mathbb{E} e^{-\mu \, {}_\eta \mathfrak{V}_t^{\nu,\delta}}
			- e^{-\mu y} \mathbb{E} e^{-\mu \, {}_\eta \mathfrak{V}_t^{\nu,\delta}} \right) m(\mathrm dy),
		\end{align}
		and equation \eqref{bilbao} becomes
		\begin{align}
			\frac{\mathrm d}{\mathrm dt} \mathbb{E} e^{-\mu\, {}_\eta \mathfrak{V}_t^{\nu,\delta}}
			= - \int_0^\infty \left( \mathbb{E} e^{-\mu \, {}_\eta \mathfrak{V}_t^{\nu,\delta}}
			- e^{-\mu y} \mathbb{E} e^{-\mu \, {}_\eta \mathfrak{V}_t^{\nu,\delta}} \right) m(\mathrm dy).
		\end{align}
		Now, a simple inversion of the Laplace transforms leads to
		\begin{align}
			\frac{\mathrm d}{\mathrm dt} \mathbb{P} \bigl( {}_\eta \mathfrak{V}_t^{\nu,\delta} \in \mathrm dx \bigr) /\mathrm dx
			= - \int_0^\infty \! \! \! \left( \mathbb{P} \bigl( {}_\eta \mathfrak{V}_t^{\nu,\delta} \in \mathrm dx \bigr) /\mathrm dx
			- \mathbb{P} \bigl( {}_\eta \mathfrak{V}_t^{\nu,\delta} \in \mathrm d(x-y) \bigr) /\mathrm d(x-y) \right) m(\mathrm dy).
		\end{align}
		Let us denote now $v(x,t) = \mathbb{P} \left( {}_\eta \mathfrak{V}_t^{\nu,\delta} \in \mathrm dx \right) /\mathrm dx$
		and with
		\begin{align}
			{}_\eta \Theta^{\nu,\delta}_x f(x,t) = \int_0^\infty \left( f(x,t) - f(x-y,t) \right) m(\mathrm dy)
		\end{align}
		the generating form of the operator. The probability density function $v(x,t)$
		of ${}_\eta \mathfrak{V}_t^{\nu,\delta}$ satisfies
		\begin{align}
			\frac{\mathrm d}{\mathrm dt} v(x,t) = - {}_\eta \Theta^{\nu,\delta}_x v(x,t), \qquad x \ge 0, \: t \ge 0.
		\end{align}
		
		\begin{remark}
			\label{uee}
			We know \citep{saigo,tom} that the Prabakhar derivative is
			defined, for a suitable set of functions \citep[see][for details]{saigo}, as
			\begin{align}
				\label{00tex}
				\bm{\mathrm D}_{\alpha,\eta,\zeta;0+}^\xi f (t) = D^{\eta+\theta}_{0+}
				\int_0^t (t-y)^{\theta-1} E_{\alpha,\theta}^{-\xi}\left[ \zeta (t-y)^\alpha \right]
				f(y) \, \mathrm dy,
			\end{align}
			with $\theta,\eta \in \mathbb{C}$, $\Re (\theta) > 0$, $\Re (\eta)>0$,
			$\zeta \in \mathbb{C}$, $t \ge 0$. The fractional
			derivative appearing in \eqref{00tex} is the Riemann--Liouville fractional derivative with respect to time $t$
			and
			\begin{align}
				\label{00poch}
				E_{\kappa,\varpi}^\gamma (x) = \sum_{r=0}^\infty \frac{x^r(\gamma)_r}{r! \Gamma(\kappa r + \varpi)},
				\qquad \kappa,\varpi,\gamma \in \mathbb{C}, \: \Re (\kappa)> 0,
			\end{align}
			is the generalized Mittag--Leffler function (see e.g.\ \citep{saigo}).
			The Laplace transform of the Prabakhar derivative is \citep{dovidio,saigo}
			\begin{align}
				\int_0^\infty e^{-st} \bm{\mathrm D}_{\alpha,\eta,\zeta;0+}^\xi f(t) \, \mathrm dt
				= s^\eta (1-\zeta s^{-\alpha})^\xi \tilde{f}(s), \qquad s > 0,
			\end{align}
			where $\tilde{f}$ is the Laplace transform of $f$.
			Therefore, with our choice of parameters,
			\begin{align}
				\int_0^\infty e^{-st} \bm{\mathrm D}_{\nu,\nu \delta,-\eta;0+}^\delta f(t) \, \mathrm dt
				= s^{\nu \delta} (1+\eta s^{-\nu})^\delta \tilde{f}(s) = (s^\nu +\eta)^\delta \tilde{f}(s), \qquad s > 0.
			\end{align}
			This implies that ${}_\eta \Theta^{\nu,\delta}_x = \bm{\mathrm D}_{\nu,\nu \delta,-\eta;0+}^\delta - \eta^\delta$
			and thus
			\begin{align}
				\label{hop2}
				\int_0^\infty e^{-\mu t} {}_\eta \Theta^{\nu,\delta}_x v(x,t) \, \mathrm dx = \left[ (\mu^\nu+\eta)^\delta - \eta^\delta \right]
				\tilde{v}(\mu,t).
			\end{align}
		\end{remark}
		
		\begin{remark}
			Note that an alternative representation for the operator ${}_\eta \Theta^{\nu,\delta}_x$ can be given in terms
			of Riemann--Liouville derivatives. We have
			\begin{align}
				{}_\eta \Theta^{\nu,\delta}_x = (\eta+D_{0+}^\nu)^\delta - \eta^\delta.
			\end{align}
			This can be proved by simply computing the Laplace transform, as follows:
			\begin{align}
				\int_0^\infty e^{-\mu x} \left[ (\eta+D_{0+}^\nu)^\delta - \eta^\delta \right] f(x) \, \mathrm dx.
			\end{align}
			We can perform a formal expansion by means of Newton's theorem obtaining
			\begin{align}
				\label{hop}
				\int_0^\infty e^{-\mu x} \sum_{r=1}^\infty \binom{\delta}{r} \eta^{\delta-r} D^{\nu r}_{0+} f(x) \, \mathrm dx
				= \left[ (\mu^\nu+\eta)^\delta - \eta^\delta \right] \tilde{f}(\mu).
			\end{align}
			We considered the set of functions for which the semigroup property for the Riemann--Liouville derivative holds,
			and also that $D_{0+}^0$ is the identity function. Note finally that \eqref{hop} coincides with \eqref{hop2}.
		\end{remark}

	\section{Generalization of the space-fractional Poisson process}

		\label{app}
		In this section we present a study of possible generalizations of the so-called space-fractional
		Poisson process \citep{orsingher} (see also \citep{toaldo} for other more generalized results and special cases).
		Furthermore, in the following we will see how the process introduced in
		the previous section arises rather naturally in this framework as a time-change of an independent
		homogeneous Poisson process.
		Furthermore, Section \ref{verylast} shows the effect of replacing the
		integer-order time derivative in the governing difference-differential equations
		with a non-local integro-differential operator with a three-parameter Mittag--Leffler function in the kernel,
		i.e.\ the so-called regularized Prabhakar derivative.
		
		\subsection{Space-fractional Poisson process}

			\label{spafra}		
			In order to make the paper as self-contained as possible, we summarize here the basic construction of the space-fractional Poisson
			process as it was carried out in \citep{orsingher}. The original aim was to generalize a homogeneous Poisson process
			in a fractional sense by introducing a fractional difference operator in the governing equations acting on the
			state space. The chosen operator is $(1-B)^\alpha$, $\alpha \in (0,1)$ (where $B$ is the backward-shift operator, that is,
			given a function $f_k$ depending on an index $k \in \mathbb{Z}$,
			$B(f_k)=f_{k-1}$ and $B^r(f_k)=B(B^{r-1}(f_k))$)
			which appears in the study of long memory time series.
			The introduction of the fractional
			difference operator $(1-B)^\alpha = \sum_{r=0}^\infty\binom{\alpha}{r}(-1)^r B^r$,
			implies a dependence of the probability of attaining a particular state
			to those of all the states below.
			The difference-differential equations for the state probabilities of the space-fractional Poisson process
			read
			\begin{align}
				\label{spa}
				\begin{cases}
					\frac{\mathrm d}{\mathrm dt} p_k(t) = -\lambda^\alpha(1-B)^\alpha p_k(t),
					& \alpha \in (0,1], \: \lambda>0,\\
					p_k(0) = \delta_{k,0},
				\end{cases}
			\end{align}
			where $p_k(t) = \mathbb{P}\{ N^\alpha(t) = k \}$, $k\geq 0$, $t\ge 0$, are the state probabilities of the space-fractional
			homogeneous Poisson process $N^\alpha(t)$, $t\ge 0$. If $\alpha=1$, the process $N^\alpha(t)$ reduces to the
			homogeneous Poisson process of rate $\lambda$.
			The space-fractional Poisson process possesses independent and stationary increments and $\mathbb{E}(N^\alpha(t))^h = \infty$,
			$h=1,2,\dots\,$.
			The Cauchy problem \eqref{spa} is easily solved using the probability generating function $G(u,t)$ and
			the Laplace transform. It turns out that the probability generating function can be written as
			\begin{align}
				G(u,t) = e^{-\lambda^\alpha t (1-u)^\alpha}, \qquad |u|\leq 1,
			\end{align}
			and by means of a simple Taylor expansion the state probability distribution is recognized as a discrete stable
			distribution and reads
			\begin{align}
				\label{sv}
				p_k(t) = \frac{(-1)^k}{k!} \sum_{r=0}^\infty \frac{(-\lambda^\alpha t)^r}{r!} \frac{\Gamma(\alpha r+1)}{
				\Gamma(\alpha r+1-k)}, \qquad k \ge 0.
			\end{align}
			
			The behaviour of the process is made apparent by the subordination representation
			\begin{align}
				\label{poldo}
				N^\alpha(V^\gamma_t) \overset{\text{d}}{=} N^{\alpha \gamma}(t), \qquad \alpha,\gamma \in (0,1).
			\end{align}
			For $\alpha=1$, the time-changed process $N^1(V^\gamma_t)$ increases super-linearly with jumps of any integer size.
			Following the same lines, in the next section we will construct a generalized model by suitably adapting the
			operator acting on the state-space.

		\subsection{Generalized model}

			\label{genspafra}
			Consider the equations
			\begin{align}
				\label{ham}
				\begin{cases}
					\frac{\mathrm d}{\mathrm dt} p_k(t) = -\left\{ \left[ \eta+\lambda^\nu(1-B)^\nu
					\right]^\delta -\eta^\delta \right\} p_k(t), & \nu n \in (0,1), \: n= \lceil \delta \rceil, \: \delta,\eta \in \mathbb{R}^+,
					\: \lambda > 0, \\
					p_k(0) = \delta_{k,0}. 
				\end{cases}
			\end{align}
			The generating function $G(u,t) = \sum_{k=0}^\infty u^kp_k(t)$ can be easily deduced from \eqref{ham} by considering the following steps.
			\begin{align}
				\label{moski}
				\frac{\partial}{\partial t} G(u,t) & = - \sum_{r=1}^\infty \binom{\delta}{r} \eta^{\delta-r} \lambda^{\nu r}
				\sum_{m=0}^\infty \binom{\nu r}{m} (-1)^m \sum_{k=m}^\infty u^k p_{k-m}(t) \\
				& = - \sum_{r=1}^\infty \eta^{\delta-r} \binom{\delta}{r} \lambda^{\nu r}
				\sum_{m=0}^\infty \binom{\nu r}{m} (-1)^m \sum_{k=0}^\infty u^{k+m} p_k(t) \notag \\
				& = - \sum_{r=1}^\infty \binom{\delta}{r} \eta^{\delta-r} \lambda^{\nu r}
				\sum_{m=0}^\infty \binom{\nu r}{m} (-u)^m G(u,t) \notag \\
				& = - G(u,t) \sum_{r=1}^\infty \binom{\delta}{r} \eta^{\delta-r} \lambda^{\nu r} (1-u)^{\nu r} \notag \\
				& = - G(u,t) \left[ \left( \eta+\lambda^\nu(1-u)^\nu \right)^\delta-\eta^\delta \right], \qquad G(u,0)=1. \notag
			\end{align}
			The solution is thus
			\begin{align}
				\label{solmamady}
				G(u,t) = e^{-t\left[ \left( \eta +\lambda^\nu(1-u)^\nu \right)^\delta -\eta^\delta \right]}, \qquad |u| \le 1.
			\end{align}

			The following theorem describes the structure of the point process for which the state probability distribution
			satisfies \eqref{ham}. For the proof we will make use of results described in Section \ref{lenovo}.
	
			\begin{theorem}
				\label{presse}
				Let ${}_\eta \mathfrak{V}_t^{\nu,\delta}$, $t \ge 0$, be the process in \eqref{proc} and
				$N(t)$, $t \ge 0$, be a homogeneous Poisson process of rate $\lambda>0$ independent of ${}_\eta \mathfrak{V}_t^{\nu,\delta}$.
				The probabilities $\mathbb{P}\{ N( {}_\eta \mathfrak{V}_t^{\nu,\delta} )=k \}$, $k \ge 0$, $t \ge 0$,
				of being in state $k$ at time $t$,
				are solutions to \eqref{ham}.
	
				\begin{proof}
					By exploiting the subordination property, the probability generating function
					of the process $N( {}_\eta \mathfrak{V}_t^{\nu,\delta} )$
					can be written as
					\begin{align}
						\int_0^\infty \mathbb{E} u^{N(s)} \mathbb{P}\bigl( {}_\eta \mathfrak{V}_t^{\nu,\delta} \in \mathrm ds\bigr).
					\end{align}
					Then, by using Proposition \ref{mamady}, we obtain the probability generating function \eqref{solmamady}.
				\end{proof}
			\end{theorem}
			
			The form of the probability generating function shows that the mean value of the associated process is infinite
			unless the subordinating process reduces to the tempered stable subordinator, i.e.\ if $\nu=1$, $\delta \in (0,1)$.
			From the probability generating function $G(u,t)$, it is possible, by a simple Taylor expansion, to obtain
			the state probability distribution $p_k(t) = \mathbb{P}[N( {}_\eta \mathfrak{V}_t^{\nu,\delta}) = k]$, $k \ge 0$, $t \ge 0$.
			We have the following theorem.
			
			\begin{theorem}
				\label{carignano}
				The state probability distribution $p_k(t) = \mathbb{P}[N( {}_\eta \mathfrak{V}_t^{\nu,\delta}) = k]$, $k \ge 0$, $t \ge 0$,
				reads
				\begin{align}
					\label{blabla}
					p_k(t) & = e^{t\eta^\delta} \frac{(-1)^k}{k!} \sum_{m=0}^\infty \frac{\Gamma(\nu m+1)}{m! \Gamma(\nu m-k+1)}
					\left( \frac{\lambda^{\nu}}{\eta} \right)^m \sum_{r=0}^\infty
					\frac{(-t\eta^\delta)^r \Gamma(r\delta +1)}{r! \Gamma(r\delta -m+1)} \\
					& = e^{t\eta^\delta} \frac{(-1)^k}{k!} \sum_{m=0}^\infty \frac{\Gamma(\nu m+1)}{m! \Gamma(\nu m-k+1)}
					\left( \frac{\lambda^{\nu}}{\eta} \right)^m
					\! \! {}_1\psi_1 \left[ \left.
					\begin{array}{l}
						(1,\delta) \\
						(1-m,\delta)
					\end{array}
					\right| -\eta^\delta t \right], \notag
				\end{align}
				where ${}_h\psi_j(z)$ is the generalized Wright function (see \citep{kilbas}, page 56,
				formula (1.11.14)).
				\begin{proof}
					Starting from the probability generating function $G(u,t)$, $t \ge 0$, $|u|\le 1$, we have
					\begin{align}
						G(u,t) = {} & e^{t \eta^\delta} e^{-t[\eta+\lambda^\nu(1-u)^\nu]^\delta}
						= e^{t \eta^\delta} \sum_{r=0}^\infty \frac{[-t(\eta+\lambda^\nu(1-u)^\nu)^\delta]^r}{r!} \\
						= {} & e^{t \eta^\delta} \sum_{r=0}^\infty \frac{(-t)^r}{r!} \sum_{m=0}^\infty \binom{r\delta}{m}
						\eta^{\delta r-m} \lambda^{\nu m}(1-u)^{\nu m} \notag \\
						= {} & e^{t\eta^\delta} \sum_{r=0}^\infty \frac{(-t)^r}{r!} \sum_{m=0}^\infty \binom{r\delta}{m}
						\eta^{\delta r-m} \lambda^{\nu m} \sum_{h=0}^\infty \binom{\nu m}{h} (-u)^h \notag \\
						= {} & \sum_{h=0}^\infty u^h e^{t \eta^\delta} \frac{(-1)^h}{h!} \sum_{r=0}^\infty \frac{(-t)^r}{r!}
						\sum_{m=0}^\infty \frac{\Gamma(r\delta +1)}{m!\Gamma(r\delta-m+1)} \eta^{\delta r-m} \lambda^{\nu m}
						\frac{\Gamma(\nu m+1)}{\Gamma(\nu m-h+1)} \notag \\
						= {} & \sum_{h=0}^\infty u^h \left\{ e^{t\eta^\delta} \frac{(-1)^h}{h!}
						\sum_{m=0}^\infty \frac{\Gamma(\nu m+1)}{m! \Gamma(\nu m-h+1)}
						\lambda^{\nu m} \eta^{-m} \sum_{r=0}^\infty \frac{(-t)^r \eta^{\delta r} \Gamma(r\delta +1)}{r! \Gamma(r\delta
						-m+1)} \right\}, \notag
					\end{align}
					and thus the claimed formula \eqref{blabla}.
				\end{proof}
			\end{theorem}
			
			\begin{remark}
				For $\delta=1$ we obtain the discrete stable distribution (2.15) of \citep{orsingher} characterizing the
				behaviour of the space-fractional Poisson process. Indeed it suffices to
				compute
				\begin{align}
					\eta^{-m} e^{t\eta} \sum_{r=0}^\infty \frac{(-t \eta)^r \Gamma(r+1)}{r!\Gamma(r-m+1)}
					= \eta^{-m} e^{t\eta} \sum_{r=m}^\infty \frac{(-t\eta)^r}{(r-m)!}
					= \eta^{-m} e^{t\eta} (-t\eta)^m e^{-t\eta} = (-t)^m.
				\end{align}
				See also \citep{toaldo}, Section 4.1.
				For general information on discrete stable random variables, the reader can refer to \citep{steutel} or \citep{devroye}.
				For $\nu=1$, $\delta \in (0,1)$,
				a time-change given by a tempered $\nu$-stable subordinator is retrieved (see \citep{toaldo}, Section 4.2).
			\end{remark}
			
			\begin{remark}
				Since $N(t)$ is a L\'evy process, the subordinated process $N( {}_\eta \mathfrak{V}_t^{\nu,\delta})$
				is a L\'evy process with Laplace exponent
				$\Phi(\mu) = (\eta+\lambda^\nu(1-e^{-\mu})^\nu)^\delta - \eta^\delta$,
				and thus it possesses stationary and independent increments. Moreover $N( {}_\eta \mathfrak{V}_t^{\nu,\delta})$
				is not in general a renewal process. This comes simply from \citep{kingman} (see also \citep{grandell}) and from the fact that,
				since ${}_\eta \mathfrak{V}_t^{\nu,\delta}$ is a properly rescaled
				and time-changed linear combination of independent stable subordinators, its inverse process
				$\mathfrak{E}_t = \inf \{w>0 \colon {}_\eta \mathfrak{V}_w^{\nu,\delta}>t) \}$ has in general non-stationary
				and dependent increments.
			\end{remark}
			
			\begin{remark}
				The reader can compare the state probabilities \eqref{blabla} with the results obtained
				in \citep{toaldo}, Section 2, with a suitably specialized Bernstein function $f$.
			\end{remark}

			\subsection{Time-fractional generalization with regularized Prabhakar derivatives}
				
				\label{verylast}
				We saw from the above analysis (see Remark \ref{uee} and the subordination result) that the difference operator
				\eqref{ham} is in practice connected with the Prabhakar derivative. This suggests that a further generalization,
				which can be still tractable, could involve a regularized Prabhakar derivative acting in time.
				This is what was considered  in \citep{tom} in the case of the classical difference operator
				and in \citep{dovidio} for the generator of a L\'evy process.
				
				Let us first recall the definition of the regularized Prabhakar derivative.
				\begin{definition}
				    Let $\beta,\omega,\gamma,\alpha \in \mathbb{C}$, $\Re(\beta),\Re(\alpha)>0$, $n=\lceil \Re(\beta) \rceil$,
				    $f \in AC^n[0,b]$, $0 < b \le \infty$, where
	    			\begin{equation*}
						AC^{n}\left[a,b\right] =\left\{ f:\left[a,b\right] \rightarrow
						\mathbb{R}\colon\frac{\mathrm d^{n-1}}{\mathrm dx^{n-1}}f \left( x\right)
						\text{ is absolutely continuous in $[a,b]$} \right\} .
					\end{equation*}
					The regularized Prabhakar derivative is defined as
				    \begin{align}
				        {}^C\mathbf{D}^{\gamma}_{\alpha, \beta, -\omega,0^+}f(x) = \mathbf{D}^{\gamma}_{\alpha, \beta, -\omega,0^+}
				        \left[ f(x) - \sum_{k=0}^{n-1}\frac{x^k}{k!}f^{(k)}(0^+) \right].
				    \end{align}
				\end{definition}
				For further details the reader can consult \citep{tom}.
				
				By means of the regularized Prabhakar derivative we can construct the following Cauchy problem.
				\begin{align}
					\label{pham}
					\begin{cases}
						{}^C\mathbf{D}^{\gamma}_{\alpha, \beta, -\omega,0^+} p_k(t) = -\left\{ \left[ \eta+\lambda^\nu(1-B)^\nu
						\right]^\delta -\eta^\delta \right\} p_k(t), & \nu n \in (0,1), \: n= \lceil \delta \rceil,
						\: \delta,\eta \in \mathbb{R}^+, \\
						p_k(0) = \delta_{k,0},
					\end{cases}
				\end{align}
				with the constraints $\omega > 0$, $\gamma \ge 0$, $0 < \alpha \le 1$, $0<\beta \le 1$.
				We also have $0 < \beta \lceil \gamma \rceil/\gamma - r\alpha < 1$, $\forall \: r=0,\dots,\lceil \gamma \rceil$, if $\gamma \ne 0$.
				With similar calculations as for \eqref{moski}, we obtain for $G(u,t) = \sum_k u^k p_k(t)$
				\begin{align}
					\label{pham2}
					\begin{cases}
						{}^C\mathbf{D}^{\gamma}_{\alpha, \beta, -\omega,0^+} G(u,t) = -\left\{ \left[ \eta+\lambda^\nu(1-u)^\nu
						\right]^\delta -\eta^\delta \right\} G(u,t), & \nu n \in (0,1), \\
						G(u,0) = 1.
					\end{cases}
				\end{align}
				By applying the Laplace transform with respect to time $t$ we have that
				\begin{align}
					\label{meryl}
					\tilde{G} (u,s) = \frac{s^{\beta-1}(1+\omega s^{-\alpha})^\gamma}{s^{\beta}(1+\omega s^{-\alpha})^\gamma
					+(\eta+\lambda^\nu(1-u)^\nu)^\delta -\eta^\delta},
				\end{align}
				which can be written, for $|[(\eta+\lambda^\nu(1-u)^\nu)^\delta -\eta^\delta]/[s^\beta(1+\omega s^{-\alpha})^\gamma]|<1$,
				\begin{align}
					\tilde{G} (u,s) = \sum_{n=0}^\infty \left\{ - [(\eta+\lambda^\nu(1-u)^\nu)^\delta -\eta^\delta] \right\}^n
					s^{-\beta n -1} (1+\omega s^{-\alpha})^{-n\gamma}.
				\end{align}
				We can now invert term by term the Laplace transform by using result (2.19) of \citep{saigo}
				and Theorem 30.1 of \citep{doetsch}.
				The probability generating function can be written as
				\begin{align}
					\label{exp}
					G(u,t) = \sum_{n=0}^\infty \left\{ - [(\eta+\lambda^\nu(1-u)^\nu)^\delta -\eta^\delta] t^\beta \right\}^n
					E_{\alpha,\beta n+1}^{\gamma n} (-\omega t^\alpha).
				\end{align}
				Note that formula (2.20) of \citep{toaldo} and equation \eqref{exp} only coincide when $\gamma=0$ in \eqref{exp}
				and, at the same time, in (2.20) of \citep{toaldo}, $f(\cdot)$ is specialized to $(\eta+\lambda^\nu(\cdot)^\nu)^\delta -\eta^\delta$.
								
				By expanding the probability generating function
				\eqref{exp} we can derive the state probability distribution for this generalized model.
				Indeed we have
				\begin{align}
					G(u,t) = \sum_{n=0}^\infty (-t^\beta)^n E_{\alpha,\beta n +1}^{\gamma n} (-\omega t^\alpha)
					\sum_{r=0}^n (-\eta^\delta)^{n-r} \sum_{m=0}^\infty \binom{r\delta}{m} \eta^{\delta r-m} \lambda^{\nu m}
					\sum_{h=0}^\infty \binom{\nu m}{h} (-u^h).
				\end{align}
				Rearranging, we get
				\begin{align}
					\label{pra}
					p_k(t) = \frac{(-1)^k}{k!} \sum_{n=0}^\infty (-t^\beta)^n E_{\alpha,\beta n+1}^{\gamma n}
					(-\omega t^\alpha) \sum_{r=0}^n (-\eta^\delta)^{n-r} \sum_{m=0}^\infty \binom{r\delta}{m} \eta^{\delta r-m}
					\lambda^{\nu m} \frac{\Gamma(\nu m +1)}{\Gamma(\nu m -k +1)}. 
				\end{align}
				
				\begin{remark}
					Note that for $\gamma=0$, $\delta=1$, the generating function \eqref{exp} reduces to that of the space-time
					fractional Poisson process \citep{orsingher} of parameters $(\beta,\nu)$,
					while for $\gamma=0$, $\nu =1$, $\delta \in (0,1)$, we obtain
					that of the tempered space-time fractional Poisson process.
					If we set $\gamma=0$, $\delta=1$, $\nu=1$, we get the time-fractional Poisson process of parameter $\beta$;
					for $\gamma=0$, $\delta=1$, $\beta=1$, we obtain the space-fractional Poisson process of parameter $\nu$.
				\end{remark}

				We now show that the probabilities $p_k(t)$ of \eqref{pra} are in fact state probabilities of a suitably
				time-changed homogeneous Poisson process.
				Consider the stochastic process, given as a sum of subordinated
				independent stable subordinators
				\begin{align}
					\textfrak{V}_t = \sum_{r=0}^{\lceil \gamma \rceil} V_{\Phi(t)}^{\beta
					\frac{\lceil \gamma \rceil}{\gamma}-r\alpha}, \qquad t \ge 0.
				\end{align}
				The random time
				change is defined as
				$\Phi(t) = \binom{\lceil \gamma \rceil}{r} V_t^{\frac{\gamma}{\lceil
				\gamma \rceil}}$,
				where $V_t^{\frac{\gamma}{\lceil \gamma \rceil}}$ is a stable
				subordinator independent of all the others and where
				$0<\beta \lceil \gamma \rceil/\gamma-r\alpha <1$ holds for each
				$r=0,1,\dots,\lceil \gamma \rceil$. The hitting time process can be defined in turn as
				\begin{align}
					\label{inverse}
					\mathfrak{U}_t = \inf \{ s \ge 0 \colon \textfrak{V}_s > t \}, \qquad t \ge	0.
				\end{align}
					
				We are now ready to state the following theorem.

				\begin{theorem}
					Let $\mathfrak{U}_t$, $t \ge 0$, be the hitting-time process in
					formula \eqref{inverse}.  Furthermore let
					$N( {}_\eta \mathfrak{V}_t^{\nu,\delta} )$
					be a
					homogeneous Poisson process of parameter $\lambda>0$, subordinated by
					the process ${}_\eta \mathfrak{V}_t^{\nu,\delta}$, defined in \eqref{proc} and independent of $%
					\mathfrak{U}_t$. The time-changed process
					\begin{align}
						\textswab{N}(t) = N\left({}_\eta \mathfrak{V}_{\mathfrak{U}_t}^{\nu,\delta}\right), \qquad t \ge 0,
					\end{align}
					has state probabilities \eqref{pra}.
				\end{theorem}
				\begin{proof}
					The claimed result can be proved by writing the probability generating
					function related to the time-changed process $\textswab{N}(t)$
					as
					\begin{align}
						\sum_{k=0}^\infty u^k \mathbb{P}(\textswab{N}(t)=k) = \int_0^\infty
						e^{-y[(\eta+\lambda^\nu(1-u)^\nu)^\delta -\eta^\delta]} \mathbb{P} (\mathfrak{U}_t \in \mathrm dy).
					\end{align}
					Therefore, by taking the Laplace transform with respect to time and taking into consideration
					Theorem 2.2 of \citep{dovidio} we have
					\begin{align}
						\int_0^\infty \int_0^\infty e^{-y[(\eta+\lambda^\nu(1-u)^\nu)^\delta -\eta^\delta] -st}
						\mathbb{P} (\mathfrak{U}_t \in
						\mathrm dy) \mathrm dt =
						\frac{s^{\beta-1}(1+\omega s^{-\alpha})^\gamma}{s^{\beta}(1+\omega s^{-\alpha})^\gamma
						+(\eta+\lambda^\nu(1-u)^\nu)^\delta -\eta^\delta},
					\end{align}
					which coincides with \eqref{meryl}.
				\end{proof}
				
				As a byproduct of our analysis we obtain
				\begin{align}
					\mathbb{E} e^{-\mu \, {}_\eta \mathfrak{V}_{\mathfrak{U}_t}^{\nu,\delta}}
					= \sum_{k=0}^\infty \left( -\left[ (\eta +\lambda^\nu\mu^\nu)^\delta-\eta^\delta \right] t^\beta  \right)^k
					E_{\alpha, \beta k+1}^{\gamma k}(-\omega t^\alpha), \qquad t \ge 0, \: s > 0,
				\end{align}
				that generalizes formula (3.37) of \citep{beghin2014fractional} with the Laplace exponent suitably specialized.

	\subsubsection*{Acknowledgments}
			Enrico Scalas acknowledges support from an SDF grant
			at the University of Sussex, UK.			
				
%%%%%%%%%%%%%%%%%%%%%%%%%%%%%%%%%%%%%%%%%%%%%%%%%%%%%%%%%%%%%%%%%%%
%%                                                               %%
%% Use the two commands below for producing your bibliography    %%
%% with bibtex, then comment again the commands and include the  %%
%% content of the .bbl file in this file below the commands.     %%
%%                                                               %%
%%%%%%%%%%%%%%%%%%%%%%%%%%%%%%%%%%%%%%%%%%%%%%%%%%%%%%%%%%%%%%%%%%%

%\bibliographystyle{amsplain}
%\bibliography{revised}

\begin{thebibliography}{10}

\bibitem{baeumer}
B~Baeumer and MM~Meerschaert, \emph{{Tempered stable L\'evy motion and
  transient super-diffusion}}, Journal of Computational and Applied Mathematics
  \textbf{233} (2010), no.~10, 2438--2448.

\bibitem{barabasi}
A-L Barabasi, \emph{Bursts: The hidden pattern behind everything we do}, Dutton
  Adult, 2010.

\bibitem{beghin2014fractional}
L~Beghin and M~{D'O}vidio, \emph{Fractional {P}oisson process with random
  drift}, Electron. J. Probab. \textbf{19} (2014), no. 122, 26.

\bibitem{beghin}
L~Beghin and C~Macci, \emph{{Alternative forms of compound fractional Poisson
  processes}}, Abstract and Applied Analysis \textbf{2012} (2012).

\bibitem{bertoin}
J~Bertoin, \emph{{L\'evy processes}}, Cambridge University Press, 1998.

\bibitem{bingham}
NH~Bingham, \emph{{Limit theorems for occupation times of Markov processes}},
  Probability Theory and Related Fields \textbf{17} (1971), no.~1, 1--22.

\bibitem{bl}
SI~Boyarchenko and S~Levendorskii, \emph{{Non-Gaussian Merton--Black--Scholes
  Theory}}, Advanced Series on Statistical Science and Applied Probability,
  vol.~9, World Scientific Singapore, 2002.

\bibitem{cat}
H~C\^{a}teau and A~Reyes, \emph{Relation between single neuron and population
  spiking statistics and effects on network activity}, Physical Review Letters
  \textbf{96} (2006), no.~5, 058101.

\bibitem{cinlar}
E~\c{C}inlar and RA~Agnew, \emph{{On the Superposition of Point Processes}},
  Journal of the Royal Statistical Society. Series B \textbf{30} (1968), no.~3,
  576--581.

\bibitem{chen}
LHY Chen and A~Xia, \emph{{Poisson process approximation for dependent
  superposition of point processes}}, Bernoulli \textbf{17} (2011), no.~2,
  530--544.

\bibitem{cox1954superposition}
DR~Cox and WL~Smith, \emph{{On the superposition of renewal processes}},
  Biometrika \textbf{41} (1954), no.~1--2, 91--99.

\bibitem{deger}
M~Deger, M~Helias, C~Boucsein, and S~Rotter, \emph{{Statistical properties of
  superimposed stationary spike trains}}, Journal of Computational Neuroscience
  \textbf{32} (2012), no.~3, 443--463.

\bibitem{devroye}
L~Devroye, \emph{{A triptych of discrete distributions related to the stable
  law}}, Statistics \& Probability Letters \textbf{18} (1993), 349--351.

\bibitem{doetsch}
G~Doetsch, \emph{{Introduction to the Theory and Application of the Laplace
  Transformation}}, Springer, Berlin, 1974.

\bibitem{dovidio}
M~D'Ovidio and F~Polito, \emph{{Fractional Diffusion-Telegraph Equations and
  their Associated Stochastic Solutions}}, [math.PR] arXiv:1307.1696 (2013).

\bibitem{franken}
P~Franken, \emph{{A Refinement of the Limit Theorem for the Superposition of
  Independent Renewal Processes}}, Teor. Veroyatnost. i Primen. \textbf{8}
  (1963), 320--328.

\bibitem{tom}
R~Garra, Gorenflo R, Polito F, and Tomovski \v{Z}, \emph{{Hilfer--Prabhakar
  Derivatives and Some Applications}}, {Applied Mathematics and Computation}
  \textbf{242} (2014), 576--589.

\bibitem{grandell}
J~Grandell, \emph{{Doubly stochastic Poisson processes}}, Springer-Verlag,
  1976.

\bibitem{grige}
B~Grigelionis, \emph{{On the Convergence of Sums of Random Step Processes to a
  Poisson Process}}, Teor. Veroyatnost. i Primen. \textbf{8} (1963), 189--194.

\bibitem{hohn}
N~Hohn and AN~Burkitt, \emph{Shot noise in the leaky integrate-and-fire
  neuron}, Physical Review E \textbf{63} (2001), 031902.

\bibitem{jiang}
Z-Q Jiang, W-J Xie, M-X Li, B~Podobnik, W-X Zhou, and H~E Stanley,
  \emph{Calling patterns in human communication dynamics}, Proceedings of the
  National Academy of Sciences \textbf{110} (2013), 1600.

\bibitem{saigo}
AA~Kilbas, M~Saigo, and RK~Saxena, \emph{{Generalized Mittag--Leffler function
  and generalized fractional calculus operators}}, {Integral Transforms and
  Special Functions} \textbf{15} (2004), no.~1, 31--49.

\bibitem{kilbas}
AA~Kilbas, HM~Srivastava, and JJ~Trujillo, \emph{{Theory and Applications of
  Fractional Differential Equations}}, Elsevier Science, 2006.

\bibitem{kingman}
JFC Kingman, \emph{{On doubly stochastic Poisson processes}}, Proc. Camb. Phil.
  Soc, vol.~60, Cambridge University Press, 1964, p.~923.

\bibitem{kop}
I~Koponen, \emph{{Analytic approach to the problem of convergence of truncated
  L{\'e}vy flights towards the Gaussian stochastic process}}, Physical Review E
  \textbf{52} (1995), no.~1, 1197.

\bibitem{kyprianou}
AE~Kyprianou, \emph{{Introductory lectures on fluctuations of L\'evy processes
  with applications}}, Springer, 2007.

\bibitem{lageras}
AN~Lageras, \emph{{A renewal-process-type expression for the moments of inverse
  subordinators}}, Journal of Applied Probability \textbf{42} (2005),
  1134--1144.

\bibitem{laskin}
N~Laskin, \emph{Fractional poisson process}, Communications in Nonlinear
  Science and Numerical Simulation \textbf{8} (2003), no.~3--4, 201--213.

\bibitem{lindner}
B~Lindner, \emph{{Superposition of many independent spike trains is generally
  not a Poisson process}}, Physical Review E \textbf{73} (2006), 022901.

\bibitem{mainardivietnam}
F~Mainardi, R~Gorenflo, and E~Scalas, \emph{A fractional generalization of the
  poisson process}, Vietnam Journal of Mathematics \textbf{32 SI} (2004),
  65--75.

\bibitem{mee}
MM~Meerschaert and H-P Scheffler, \emph{{Limit theorems for continuous-time
  random walks with infinite mean waiting times}}, Journal of Applied
  Probability \textbf{41} (2004), no.~3, 623--638.

\bibitem{meerschaert}
MM~Meerschaert and A~Sikorskii, \emph{{Stochastic Models for Fractional
  Calculus}}, vol.~43, de Gruyter, 2011.

\bibitem{orsingher}
E~Orsingher and F~Polito, \emph{{The space-fractional Poisson process}},
  Statistics \& Probability Letters \textbf{82} (2012), no.~4, 852--858.

\bibitem{toaldo}
E~Orsingher and B~Toaldo, \emph{{Counting Processes with Bern\v{s}tein
  Intertimes and Random Jumps}}, To appear in J. Appl. Probab., [math.PR]
  arXiv:1312.1498 (2013).

\bibitem{podlubny}
I~Podlubny, \emph{{Fractional differential equations}}, vol. 198, Academic
  press, 1998.

\bibitem{kaizoji}
M~Politi, T~Kaizoji, and E~Scalas, \emph{Full characterization of the
  fractional poisson process}, Europhysics Letters \textbf{96} (2011), no.~2,
  20004.

\bibitem{rosinski}
J~Rosi{\'n}ski, \emph{{Tempering stable processes}}, Stochastic processes and
  their applications \textbf{117} (2007), no.~6, 677--707.

\bibitem{shimokawa}
T~Shimokawa, A~Rogel, K~Pakdaman, and S~Sato, \emph{Stochastic resonance and
  spike-timing precision in an ensemble of leaky integrate and fire neuron
  models}, Physical Review E \textbf{59} (1999), 3461.

\bibitem{steutel}
FW~Steutel and K~van Harn, \emph{{Discrete Analogues of Self-Decomposability
  and Stability}}, The Annals of Probability \textbf{7} (1979), no.~5,
  893--899.

\bibitem{veillette}
M~Veillette and MS~Taqqu, \emph{{Using differential equations to obtain joint
  moments of first-passage times of increasing L\'evy processes}}, Statistics
  \& Probability Letters \textbf{80} (2010), 697--705.

\end{thebibliography}

% add below the content of your .bbl file produced by bibtex.

\providecommand{\bysame}{\leavevmode\hbox to3em{\hrulefill}\thinspace}
\providecommand{\MR}{\relax\ifhmode\unskip\space\fi MR }
% \MRhref is called by the amsart/book/proc definition of \MR.
\providecommand{\MRhref}[2]{%
  \href{http://www.ams.org/mathscinet-getitem?mr=#1}{#2}
}
\providecommand{\href}[2]{#2}

\end{document}